\documentclass{tac}

\usepackage[all]{xy}
\usepackage{amsmath}
\usepackage{amssymb}
\usepackage{mathtools}
\usepackage{latexsym}

\newcommand*{\into}{\rightarrowtail}
\newcommand*{\onto}{\twoheadrightarrow}
\newcommand{\defeq}{\mathrel{:=}} 

\newcommand{\op}{\mathrm{op}}

\def\topdf{\texorpdfstring}

\newcommand\Z{\mathbb Z}
\newcommand\N{\mathbb N}
\newcommand\C{\mathbb C}
\newcommand\R{\mathbb R}

\newcommand{\cC}{\mathcal C}
\newcommand{\cD}{\mathcal D}
\newcommand{\cE}{\mathcal E}

\theoremstyle{definition}
\newtheorem{definition}[equation]{Definition}

\theoremstyle{plain}
\newtheorem{proposition}[equation]{Proposition}

\DeclareMathOperator{\coker}{coker}

\DeclarePairedDelimiterX{\setgiven}[2]{\{}{\}}{#1\,{:}\,\mathopen{}#2}

\input diagxy

\usepackage[colorlinks=true]{hyperref}
\hypersetup{allcolors=[rgb]{0.1,0.1,0.4}}

\author{Guillermo Corti\~nas, Devarshi Mukherjee}

\thanks{The first named author was supported by CONICET and partially supported by grants PICT 2017-1395 from Agencia Nacional de Promoci\'on Cient\'\i fica y T\'ecnica, UBACyT 0256BA from Universidad de Buenos Aires, and PGC2018-096446-B-C21 from the Spanish Ministerio de Ciencia e Innovaci\'on. The second named author was funded by a Feodor-Lynen Fellowship of the Alexander von Humboldt Foundation. The authors thank the anonymous referee for helpful comments.}

\address{Dep. Matemática-IMAS\\
 FCEyN-UBA, Ciudad Universitaria Pab 1\\
 1428 Buenos Aires\\ 
Argentina}

\title{A Quillen model structure of local homotopy equivalences}

\keywords{Model categories, cyclic homology, functional analysis}
\amsclass{18N40, 18G35, 19D55}

\eaddress{gcorti@dm.uba.ar\CR dmukherjee@dm.uba.ar}

\newtheorem{thm}{Theorem}
\newtheorem{theorem}{Theorem} 
\newtheorem{lem}{Lemma}

\newtheoremrm{rem}{Remark}

\mathrmdef{Hom}
\mathbfdef{Set}

\begin{document}

\maketitle
\begin{abstract}
In this note, we construct a closed model structure on the category of $\Z/2\Z$-graded complexes of projective systems of ind-Banach spaces. When the base field is the fraction field \(F\) of a complete discrete valuation ring \(V\), the homotopy category of this model category is the derived category of \(\Z/2\Z\)-graded complexes of the quasi-abelian category \(\overleftarrow{\mathsf{Ind}(\mathsf{Ban}_F)}\). This homotopy category is the appropriate target of the local and analytic cyclic homology theories for complete, torsionfree \(V\)-algebras and \(\mathbb{F}\)-algebras. When the base field is \(\C\), the homotopy category is the target of local and analytic cyclic homology for pro-bornological \(\C\)-algebras, which includes the subcategory of pro-\(C^*\)-algebras. 
\end{abstract}


\section{Introduction}\label{sec-Introduction}

In their fundamental work on periodic cyclic homology, leading to the celebrated excision theorem, Cuntz and Quillen (\cite{Cuntz-Quillen:Cyclic_nonsingularity}) associate to each algebra $A$, a functorial inverse system 
$X^\infty (A)=\{X^{n+1}(A)\to X^n(A): n\ge 1\}$ of $\Z/2\Z$-graded complexes. Localizing the category $\overleftarrow{\mathsf{Kom}}$ of inverse systems of \(\Z/2\Z\)-graded complexes - or briefly, \textit{pro-supercomplexes} - at a certain class of weak equivalences, called \textit{local equivalences}, one obtains a derived category $\mathsf{Der}(\overleftarrow{\mathsf{Kom}})$ which is enriched over $\Z/2\Z$-graded complexes. The \textit{bivariant periodic cyclic homology} $HP_*(A,B)$ of a pair of algebras $(A,B)$ is then defined as the homology of the hom-complex $\Hom_{\mathsf{Der}(\overleftarrow{\mathsf{Kom}})}(X^{\infty}(A),X^{\infty}(B))$. These weak equivalences are part of a Quillen model structure on the category of pro-supercomplexes described in \cite{Cortinas-Valqui:Excision}. 

In the study of variants of $HP$ in several contexts of topological and bornological algebras, one is lead to consider inverse systems of \emph{directed} systems of $\Z/2\Z$-complexes. That is notably the case of analytic cyclic homology for torsion-free complete bornological algebras over a discrete valuation ring $V$ \cite{Cortinas-Meyer-Mukherjee:NAHA}, algebras over its residue field $\mathbb{F}$ \cite{Meyer-Mukherjee:HA} and local cyclic homology for dagger algebras \cite{Meyer-Mukherjee:localhc}. In each of the latter cases, the relevant homology is represented by a functor taking values in the category \(\mathsf{Der}(\overleftarrow{\mathsf{Ind}(\mathsf{Ban}_F)})\) which results from the category  \(\overleftarrow{\mathsf{Kom}(\mathsf{Ind}(\mathsf{Ban}_F))}\) of projective systems of $\Z/2\Z$-graded complexes of inductive systems of Banach spaces over $F$, upon inverting local weak equivalences. Similarly, local cyclic homology of pro-$C^*$-algebras can be defined in terms of a functorial complex taking values in \(\mathsf{Der}(\overleftarrow{\mathsf{Ind}(\mathsf{Ban}_\C)})\). The purpose of this article is to prove the following.

\thm\label{intro:thm} Let $\mathcal{C}$ be an exact category with enough projectives. Then the category  $\overleftarrow{\mathsf{Kom}(\mathsf{Ind}((\mathcal{C})))}$ carries an injective model structure where the weak equivalences are the local weak equivalences. Thus for the associated homotopy category we have
\[
\mathsf{Ho}(\overleftarrow{\mathsf{Kom}(\mathsf{Ind}((\mathcal{C})))})\cong \mathsf{Der}(\overleftarrow{\mathsf{Ind}(\mathcal{C})}).
\]
This applies, in particular, when $\mathcal{C}$ is the category of Banach spaces over $\R$, $\C$ or any complete valuation field $F$, equipped with the split-exact structure. If $F$ is discretely valued, the latter agrees with the quasi-abelian structure.
\endthm

\bigskip

The article is organised as follows. In Section \ref{sec:proE} we consider, for an additive category $\cE$ with kernels and cokernels, an exact structure on the category \(\mathsf{Ind}(\cE)\) whose distinguished extensions are kernel-cokernel pairs that split locally. This means that \(\mathrm{Hom}(X,-)\) preserves cokernels in \(\mathsf{Ind}(\cE)\) for \(X \in \cE\). When \(\cE\) is quasi-abelian, we use this exact structure on \(\mathsf{Ind}(\cE)\) to induce an exact structure on the category \(\overleftarrow{\mathsf{Ind}(\cE)}\) of countable projective systems of inductive systems of objects in \(\cE\). We call this exact structure the \textit{locally split} exact structure. The main result of Section \ref{sec:proE} shows that \(\overleftarrow{\mathsf{Ind}(\cE)}\) has enough injectives for the locally split exact structure. Moreover, this category is countably complete since \(\cE\) is in particular additive with kernels and cokernels. In Section \ref{sec:model}, we use the work of Gillespie \cite{gillespie2011model} and Kelly \cite{kelly2016homotopy} to show in  Proposition \ref{prop:model_pro-kom-general} that if $\mathfrak{F}$ is a countably complete exact category with enough injectives, then $\mathsf{Kom}(\mathfrak{F})$ carries a model structure, where weak equivalences are quasi-isomorphisms, and where cofibrations are degreewise inflations. In particular, this applies to $\mathfrak{F}=\overleftarrow{\mathsf{Ind}(\cE)}$ with the exact structures of Section \ref{sec:proE}.

Section \ref{sec:bivariant} specializes all of the above to the category $\cC=\mathsf{Ban}_k$, where $k$ is any nontrivially valued complete field. In Subsection \ref{subsec:nonahc} we consider the case when $k$ is discretely valued. Proposition \ref{prop:explicit_local_homotopy_equivalences} explicitly describes the weak equivalences in the resulting model structure on $\overleftarrow{\mathsf{Kom}(\mathsf{Ind}(\mathsf{Ban}_k))}$, showing that when $k$ is discretely valued, they are exactly the local homotopy equivalences used in local and analytic cyclic homology for nonarchimedean algebras \cite{Cortinas-Meyer-Mukherjee:NAHA,Meyer-Mukherjee:HA,Meyer-Mukherjee:localhc}. Thus our results allow us to interpret those homologies as homomorphism spaces in the homotopy category of our model category.  In Subsection \ref{subsec:prohl} we consider the case when $k=\C$. If one disregards the projective system level, then the exact structure on \(\mathsf{Ind}(\mathsf{Ban}_\C)\) of Section \ref{sec:proE} has previously been used in \cite[Section 2.3]{Meyer:HLHA} to define the target of local cyclic homology for locally multiplicative complex Banach algebras. The availability of a model structure for complexes of pro-ind-Banach spaces over \(\C\) means that we can extend analytic and local cyclic homology to projective systems of complete bornological and $\mathsf{Ind}$-Banach algebras, respectively, having the expected homotopy invariance, stability and excision properties (Theorem \ref{thm:pro-HL-properties}), and in the case of local cyclic homology, also invariance under isoradial embeddings (Theorem \ref{thm:isoradial-pro}). Using these properties and the universal property of Bonkat's bivariant $K$-theory for pro-$C^*$-algebras \cite{MR3649664,Bonkat:Thesis}, we obtain a Chern character from the latter to our bivariant local cyclic homology (see \ref{subsub:chern}). This Chern character could be used in the future to study the topological \(K\)-theory and local cyclic homology of the recently defined pro-\(C^*\)-algebras of noncommutative classifying spaces of quasi-topological groups, appearing in \cite{chirvasitu2023non}.

\section{An exact structure on pro-objects in \topdf{\(\mathcal{E}\)}{E}}\label{sec:proE}

In this section, we recall some generalities on Quillen's exact categories. We will show that under certain assumptions, an exact category \(\mathcal{E}\) produces the so-called injective model structure on the category of (unbounded) chain complexes \(\mathsf{Ch}(\mathcal{E})\), the homotopy category of which is the derived category of \(\mathcal{E}\). 

Let \(\mathcal{E}\) be an additive category. An \textit{extension} in \(\mathcal{E}\) is a diagram of the form 
\[
K \overset{i}\into E \overset{p}\onto Q\] where \(i\) is the kernel of \(p\) and \(p\) is the cokernel of \(i\). An \textit{exact category} \footnote{This is an equivalent formulation of Quillen's original definition due to Keller (see Appendix A in \cite{Keller:Chain_stable}).} is an additive category with a distinguished class of extensions, called \textit{conflations} - wherein the maps \(i\) and \(p\) are called inflations and deflations, respectively - satisfying the following properties: 

\begin{itemize}
\item the identity map on the zero object is a deflation;
\item if \(f\) and \(g\) are composable deflations, then their composition is a deflation;
\item the pullback of a deflation along an arbitrary morphism of \(\mathcal{E}\) exists and is a deflation;
\item the pushout of an inflation along an arbitrary morphism of \(\mathcal{E}\) exists and is an inflation.
\end{itemize}

In this article, our interest is a more convenient class of exact categories, called \textit{quasi-abelian} categories in the sense of \cite{Schneiders:Quasi-Abelian}. These are additive categories with kernels and cokernels, which are stable under pushout and pullback, respectively. In other words, they are exact categories whose distinguished class of extensions is the class of all kernel-cokernel pairs. Note however that being quasi-abelian is a \emph{property} of a category rather than additional structure. 

\begin{definition}\label{def:exact_complex}
We call a \(\Z\)-graded chain complex \((C,d)\) with entries in an exact category \(\mathcal{E}\) with kernels \textit{exact} if the induced diagram \[\ker(d) \into C \onto \ker(d)\] is a conflation in \(\mathcal{E}\). Here the inflation \(\ker(d) \to C\) is the canonical inclusion and the deflation \(C \to \ker(d)\) is the canonical map induced by \(d\). A chain map \(f \colon C \to D\) is called a \textit{quasi-isomorphism} if its mapping cone \(\mathsf{cone}(f)\) is exact. 
\end{definition}

We denote by \(\mathsf{Kom}(\mathcal{E})\) the category of \(\Z/2\Z\)-graded  chain complexes (also called \textit{supercomplexes}) with entries in \(\mathcal{E}\). Its internal \(\mathrm{Hom}\) is defined as the \textit{mapping complex}  \(\mathrm{HOM}_{\mathcal{E}}(C,D) \in \mathsf{Kom}(\mathcal{E})\) for two complexes \(C\), \(D \in \mathsf{Kom}(\mathcal{E})\) is defined as
  \begin{align*}
\mathrm{HOM}_{\mathcal{E}}(C,D)_n \defeq \underset{k \in \Z/2\Z}\prod \Hom_{\mathcal{E}}(C_k, D_{k+n}), \\
\delta_n((f_k)_{k \in \Z/2\Z}) = \delta_{k+n}^D \circ f_k - (-1)^n f_{k-1} \circ \delta_k^C,  
\end{align*}
for \(C\), \(D \in \mathsf{Kom}(\mathcal{E})\), and \(n = 0,1\). This definition makes sense for chain complexes in any additive category. As we are interested in cyclic homology theories which are \(2\)-periodic, we restrict ourselves to this category rather than working in the category \(\mathsf{Ch}(\mathcal{E})\) of \(\Z\)-graded chain complexes. The homotopy category \(\mathsf{HoKom}(\mathcal{E})\) of the category \(\mathsf{Kom}(\mathcal{E})\) is a triangulated category. We define the \textit{derived category} \(\mathsf{Der}(\mathcal{E})\) of an exact category \(\mathcal{E}\) as  the localisation of the homotopy category of \(\Z/2\Z\)-graded chain complexes \(\mathsf{HoKom}(\mathcal{E})\) at the quasi-isomorphisms. 

Given a (locally small) category \(\mathcal{C}\), we denote by \(\overleftarrow{\mathcal{C}}\) the category of countable projective systems (or briefly, \textit{pro-systems})  over \(\mathcal{C}\). Given two such pro-systems \(X\) and \(Y\), we define its Hom-set as \[\Hom_{\overleftarrow{\mathcal{C}}}(X,Y) = \underset{n}\varprojlim  \underset{m}\varinjlim \Hom_{\mathcal{C}}(X_m, Y_n).\] We will proceed as in \cite{Cortinas-Valqui:Excision} to construct a model category structure on the category of pro-supercomplexes \(\overleftarrow{\mathsf{Kom}(\mathcal{E})}\).

\lem\label{lem:quasi-abelian-pro}
Let \(\mathcal{E}\) be a quasi-abelian category. Then \(\overleftarrow{\mathcal{E}}\) is a quasi-abelian category. 
\endlem

\begin{proof}
The same proof as in \cite[Proposition 7.1.5]{Prosmans:Derived_limits} works for countable projective systems.  
\end{proof}

In our applications to local cyclic homology, although the underlying category \(\mathcal{E}\) is usually quasi-abelian, we use an exact category structure to do homological algebra which we now describe. Let \(\mathcal{C}\) be an additive category with kernels and cokernels. Its category of inductive systems \(\mathsf{Ind}(\mathcal{C})\) is the category of functors \(I \to \mathcal{C}\), where \(I\) is a filtered category. For two such inductive systems \(X \colon I \to \mathcal{C}\) and \(Y \colon J \to \mathcal{C}\), the morphism set is defined as the set 
\[\mathrm{Hom}(X,Y) = \varprojlim_i \varinjlim_j \Hom_{\mathcal{C}}(X_i, Y_j).\] We may equip \(\mathcal{C}\) with the split exact structure, that is, an extension \(K \into E \onto Q\) is a conflation in \(\mathcal{C}\) if and only if for each \(X \in \mathcal{C}\), \[\mathrm{Hom}_{\mathcal{C}}(X, K) \into \mathrm{Hom}_{\mathcal{C}}(X, E) \onto \mathrm{Hom}_{\mathcal{C}}(X, Q)\] is an exact sequence of abelian groups. This induces the following exact category structure on \(\mathsf{Ind}(\mathcal{C})\):

\begin{definition}\label{def:locally_split_exact}
Let \(\mathcal{C}\) be an additive category with kernels and cokernels. We say an extension \(K \into E \onto Q\) in \(\mathsf{Ind}(\mathcal{C})\) is \textit{ind-locally split} if for every \(X \in \mathcal{C}\), the induced sequence \[\Hom_{\mathsf{Ind}(\mathcal{C})}(X,K) \into \Hom_{\mathsf{Ind}(\mathcal{C})}(X,E) \onto \Hom_{\mathsf{Ind}(\mathcal{C})}(X,Q)\] is a short exact sequence of abelian groups.  Equivalently, every morphism \(X \to Q\) with \(X\in \cC\) in \(\mathsf{Ind}(\mathcal{C})\) lifts to a morphism \(X \to E\).
\end{definition}

\lem\label{lem:ind-locally-split-exact}
Let \(\cC\) be a quasi-abelian category. Then the ind-locally split extensions of Definition \ref{def:locally_split_exact} form an exact category structure on \(\mathsf{Ind}(\cC)\).
\endlem

\begin{proof}
The identity map on the zero object is clearly a deflation. To check that the composition of deflations is a deflation, consider two such deflations \(X \onto Y\) and \(Y \onto Z\). Then their composition is a cokernel. Now let \(P \to Z\) be a morphism, where \(P \in \cC\). Since \(Y \onto Z\) is a deflation, there is a lifting \(P \to Y\), and since \(X \onto Y\) is a deflation, there is a lifting \(P \to X\), as required. We now check that the pullback of a deflation \(E \onto Q\) by an arbitrary morphism \(Q' \to Q\) is a deflation. Since \(\cC\) is finitely complete, \(\mathsf{Ind}(\cC)\) has pullbacks. Let \(P\) denote the pullback of the maps \(E \onto Q\) and \(Q' \to Q\). The hypothesis that \(\cE\) is quasi-abelian implies that the resulting canonical map \(P \to Q'\) is a cokernel. To see that the map \(P \to Q'\) is locally split, consider a morphism \(X \to Q'\) where \(X \in \cE\). Composing with the map \(Q' \to Q\) and using that the original cokernel \(E \to Q\) was locally split, we obtain a lifting \(X \to E\). The existence of the required  lifting \(X \to P\) follows from the fact that \(P\) is a pullback. Pushouts are dealt with dually.   
\end{proof}

\begin{remark}[Indisation of an exact category]
In Lemma \ref{lem:ind-locally-split-exact}, the hypothesis that \(\cC\) is quasi-abelian is sufficient to show that the pullback of a cokernel is a cokernel. This is also the situation that is most relevant for the purposes of the article. However, more generally, if \(\cC\) is any small exact category, we can define the category \(\mathsf{Lex}(\mathcal{C}^\op, \mathsf{Mod}_\Z)\) of left exact functors on \(\mathcal{C}\). This is an abelian category (see Section 3  of \cite{braunling2016tate}). The filtered cocompletion of the image of the Yoneda embedding \(\mathcal{C} \subseteq \mathsf{Lex}(\mathcal{C}^\op, \mathsf{Mod}_\Z)\) is precisely the category of inductive systems \(\mathsf{Ind}(\mathcal{C})\). Here by filtered cocompletion, we mean those functors in \(\mathsf{Lex}(\cC^\op, \mathsf{Mod}_\Z)\) which are direct limits of representable functors.  Furthermore, \(\mathsf{Ind}(\mathcal{C})\) is an extension closed subcategory of \(\mathsf{Lex}(\mathcal{C}^\op, \mathsf{Mod}_\Z)\). Using this identification, it is shown in \cite[Proposition 4.8]{kelly2021analytic} that \(K \into E \onto Q\) is a conflation in \(\mathsf{Ind}(\mathcal{C})\) if and only if it can be represented by a diagram \((K_i \into E_i \onto Q_i)_{i \in I}\) of conflations in \(\mathcal{C}\) for a filtered category \(I\). This is called the \textit{indisation of an exact category}.
\end{remark}

We now describe an exact category structure on \(\overleftarrow{\mathcal{E}}\) which takes this internal exact structure on \(\cE\) into account. 

\begin{definition}\label{def:internal-external-pro}
Let \(\mathcal{E}\) be an additive category with kernels and cokernels and \(J\) a full subcategory. We call an extension \(K \into E \onto Q\) in \(\overleftarrow{\mathcal{E}}\) \textit{locally split relative to \(J\)} if the induced diagram \[\mathrm{Hom}_{\overleftarrow{\mathcal{E}}}(X,K) \into \mathrm{Hom}_{\overleftarrow{\mathcal{E}}}(X,E) \onto \mathrm{Hom}_{\overleftarrow{\mathcal{E}}}(X,Q)\] is an extension of abelian groups for all \(X \in J\).
\end{definition}

\begin{lemma}\label{lem:internal-pro-exact}
Let \(\mathcal{E}\) be quasi-abelian and \(J\) a full subcategory. The locally split extensions relative to \(J\) yield an exact category structure on \(\overleftarrow{\mathcal{E}}\). 
\end{lemma}

\begin{proof}
The identity map is clearly a deflation. To see that the composition of deflations is a deflation, let \(E \onto Q\) and \(Q \onto L\) be two such deflations, and let \(X \to L\) be a morphism with \(X \in J\). Since \(Q \onto L\) is a deflation, we get a lifting \(X \to Q\), and since \(E \onto Q\) is a deflation, we get the required lifting \(X \to E\) of the composition \(E \to L\). To see that the pullback of a deflation \(E \onto Q\) by an arbitrary map \(L \to Q\) is a deflation, we first note that the pullback \(P \to L\) is a cokernel as \(\overleftarrow{\cE}\) is quasi-abelian. Now suppose \(X \to L\) is a morphism, where \(X \in J\). Then the composition \(X \to L  \to Q\) has a lifting to \(E\) since \(E \onto Q\) is a deflation. By the universal property of pullbacks, we get the required lifting \(X \to P\). Now let \(K \into E\) be an inflation and \(K \to L \) be an arbitrary morphism in \(\overleftarrow{\cE}\). Then the pushout \(P\) is a kernel again as \(\overleftarrow{\cE}\) is quasi-abelian, with cokernel \(Q\). That is, we have a kernel-cokernel pair \(L \into P \onto Q\). Consider a map \(X \to Q\) with \(X \in J\). Then since \(E \onto Q\) is a deflation, there is a lifting \(X \to E\), whose composition with the canonical map \(E \to P\) yields the desired lifting.   
\end{proof}

We now combine the relative locally split exact structure with the ind-locally split exact structure on \(\cE = \mathsf{Ind}(\mathcal{C})\). More concretely, we have the following:

\begin{proposition}\label{prop:locally-split-specialised}
Let \(\cC\) be an additive category with kernels and cokernels. Then an extension \(K \into E \onto Q\) in \(\overleftarrow{\mathsf{Ind}(\cC)}\) is locally split relative to \(\cC\) if and only if it is isomorphic to a diagram \((K_n \into E_n \onto Q_n)_{n \in \mathbb{N}}\) of ind-locally split extensions.
\end{proposition}

\begin{proof}
By the proof of \cite[Proposition 4.3.13]{Cortinas-Meyer-Mukherjee:NAHA}, an extension \(K \into E \onto Q\) in \(\overleftarrow{\mathsf{Ind}(\cC)}\) is isomorphic to a diagram of extensions \(K_n \into E_n \onto Q_n\) in \(\mathsf{Ind}(\cC)\). Being locally split relative to \(\cC\) means in particular that for each \(i\) in the indexing category of \(Q_n\), the canonical map \(Q_{n,i} \to Q_n\) has a lifting to \(E_n\). Combining this with the fact that for a constant pro-system \(X\), a morphism \(X \to Q\) is an inverse system of morphisms \(X \to Q_n\) implies the result.  
\end{proof}

Let \(\mathcal{E}\) be an exact category. An object \(Z \in \cE\) is said to be \textit{relatively injective} if for any inflation \(f \colon X \into Y\), the induced map \(\mathrm{Hom}_{\mathcal{E}}(f,Z) \colon \mathrm{Hom}_{\mathcal{E}}(Y,Z) \onto \mathrm{Hom}_{\mathcal{E}}(X,Z)\) is a surjection of abelian groups. An exact category is said to have \textit{enough injectives} if for any \(X \in \mathcal{E}\), there is an inflation \(X \into Z\), where \(Z\) is relatively injective. Dually, one defines \textit{projective objects} relative to the exact category structure on \(\mathcal{E}\) as those objects \(P\) for which \(\Hom(P,-)\) maps a deflation in \(\mathcal{E}\) to a surjection of abelian groups.

\begin{lemma}\cite[Proposition 7.3.2]{Prosmans:Derived_limits}\label{lem:enough-injectives}
Let \(\cE\) be a quasi-abelian category with enough injectives. Then \(\overleftarrow{\cE}\) has enough injectives. Dually, if \(\cE\) has enough projectives, then \(\mathsf{Ind}(\cE)\) has enough projectives. 
\end{lemma}

Lemma \ref{lem:enough-injectives} only provides information about injective objects for filtered cocompletions and pro-completions of quasi-abelian categories. In our main applications, however, the relevant exact structure on \(\cE = \mathsf{Ind}(\cC)\) is the ind-locally split exact structure.

\begin{lemma}\cite[Proposition 4.8]{kelly2021analytic}\label{lem:ind-C-enough-injectives}
Let \(\mathcal{C}\) be an exact category with kernels and cokernels, and enough projective objects. Then \(\mathsf{Ind}(\mathcal{C})\) is an elementary exact category for the indisation of the exact category structure of $\cC$. 
\end{lemma}

\begin{lemma}\label{lem:ind-enough-inj}
Suppose \(\mathcal{E}\) is an elementary exact category, then \(\mathcal{E}\) has enough (functorial) injectives. 
\end{lemma}

\begin{proof}
Combine \cite[Lemma 3.3.54]{kelly2016homotopy}\label{lem:elementary-Grothendieck-type} and \cite[Corollary 5.9]{stovicek2013exact}.
\end{proof}

\begin{corollary}\label{cor:ind-C-enough-injectives}
Let \(\mathcal{C}\) be a quasi-abelian category. Then \(\mathsf{Ind}(\mathcal{C})\) has enough (functorial) injectives for the ind-locally split exact structure.
\end{corollary}

\begin{proof}
We view \(\cC\) as an exact category with respect to the split exact structure, with respect to which every object is projective. Furthermore, the ind-locally split exact structure on \(\mathsf{Ind}(\mathcal{C})\) is the indisation of the split exact structure on \(\mathcal{C}\) (see \cite[Example 4.26]{kelly2021note}). So by Lemma \ref{lem:ind-C-enough-injectives}, it is elementary. The conclusion follows from Lemma \ref{lem:ind-enough-inj}.
\end{proof}

\begin{theorem}\label{thm:enough-injectives-pro}
Let \(\cC\) be a quasi-abelian category. Then \(\overleftarrow{\mathsf{Ind}(\cC)}\) with the locally split exact structure relative to \(\cC\) has enough injectives.
\end{theorem}

\begin{proof}
Let \(X \in \overleftarrow{\mathsf{Ind}(\cC)}\). Then the map \(X \to ``\prod" X\) into the fake product is an inflation. It is a kernel by the proof of \cite[Proposition 7.3.2]{Prosmans:Derived_limits}. The resulting kernel-cokernel pair can be represented by the kernel-cokernel pairs \[(X_n \into \bigoplus_{i=0}^n X_i \onto \bigoplus_{i=0}^{n-1} X_i)_{n \in \N},\] where the coproduct is taking place in the category \(\mathsf{Ind}(\cC)\). Furthermore, the cokernel splits by the obvious inclusion into the first \(n-1\)-summands, which suffices to show that the cokernel \(\bigoplus_{i=0}^n X_i \onto \bigoplus_{i=0}^{n-1} X_i\) is ind-locally split.  

Now since \(\mathsf{Ind}(\cC)\) has enough injectives for the ind-locally split exact structure by Corollary \ref{cor:ind-C-enough-injectives}, for each \(n \in \mathbb{N}\), there is an inflation \(j_n \colon X_n \into I_n\), where \(I_n\) is relatively injective. We first observe that each \(I_n\) is relatively injective when we view it as a constant pro-object in \(\overleftarrow{\mathsf{Ind}(\cC)}\). This is because if \(X \into Y\) is an inflation in \(\overleftarrow{\mathsf{Ind}(\cC)}\), and \(X \to I_n\) an arbitrary morphism, then as \(X \into Y\) is in particular a kernel in \(\overleftarrow{\mathsf{Ind}(\cC)}\), by \cite[Lemma 7.3.1]{Prosmans:Derived_limits}, there is a lifting \(Y \to I_n\). Now since \(X_n \overset{j_n}\into I_n \onto \mathrm{coker}(j_n)\) is ind-locally split for each \(n\) by hypothesis, taking products in the pro-category \(\overleftarrow{\mathsf{Ind}(\cC)}\), we get a kernel-cokernel pair \[``\prod" X \into ``\prod" I \onto ``\prod" \coker(j_n),\] which is locally split. Finally, the fake product of a sequence of relatively injective objects in relatively injective by the same argument as the proof of \cite[Lemma 2.2.2]{Cortinas-Valqui:Excision}.  
\end{proof}

\section{From exact to model categories}\label{sec:model}

In this section, we show that under suitable conditions on an exact category \(\mathfrak{F}\), there is an induced closed model structure on \(\mathsf{Kom}(\mathfrak{F})\). We then specialise this to the exact category \(\overleftarrow{\mathsf{Ind}(\cE)}\) with the locally split exact structure relative to a quasi-abelian category \(\cE\). We call an object \(X \in \mathsf{Kom}(\mathfrak{F})\) \textit{fibrant} if at each degree \(n\), \(X_n\) is relatively injective. The model structure we desire is defined as follows:

\begin{definition}\label{def:injective-model-structure}
Let \(\mathfrak{F}\) be an exact category. The \textit{injective model category structure} on \(\mathsf{Kom}(\mathfrak{F})\), if it exists, is the model structure in which
\begin{itemize}
\item weak equivalences are the quasi-isomorphisms;
\item cofibrations are the degree-wise inflations;
\item fibrations are the degree-wise deflations with fibrant kernels.
\end{itemize}
\end{definition}

We now use the general machinery developed in \cite{kelly2016homotopy} to find conditions on an exact category under which the injective model structure exists. The results we need from \cite{kelly2016homotopy} that are stated for arbitrary unbounded chain complexes work verbatim for \(\Z/2\Z\)-periodic (unbounded) chain complexes. 

\begin{lemma}\label{lem:fibration-description}
Let \(\mathfrak{F}\) be a countably complete exact category with enough injectives. A morphism \(f \colon X \to Y\) in \(\mathsf{Ch}(\mathfrak{F})\)  is degree-wise a deflation in \(\mathfrak{F}\) and its kernel is fibrant if and only if it satisfies the right lifting property with respect to cofibrations that are weak equivalences. Here cofibrations and weak equivalences refer to degree-wise inflations and quasi-isomorphisms, respectively. 
\end{lemma}

\begin{proof}
We use some terminology and notation from \cite[Section 4]{kelly2016homotopy}.  Let \(\mathcal{F}\) denote the class of morphisms in \(\mathsf{Ch}(\mathfrak{F})\) that satisfy the right lifting property with respect to cofibrations that are weak equivalences, and let \(\mathcal{F}'\) denote the class of morphisms that are degreewise deflations in \(\mathfrak{F}\), and whose kernels are relatively injective. Let \(\mathcal{C}\) and \(\mathcal{W}\) denote the class of cofibrations and weak equivalences, respectively. Since \(\mathfrak{F}\) has enough injectives, the pair \((\mathrm{Ob}(\mathfrak{F}), \mathsf{Inj}(\mathfrak{F}))\) of all objects and relatively injective objects is a complete cotorsion pair on \(\mathfrak{F}\) in the sense of \cite[Definitions 4.1.2 and 4.1.3]{kelly2016homotopy}. By \cite[Corollary 4.2.25]{kelly2016homotopy}, the pair  

\[(\widetilde{\mathrm{Ob}(\mathfrak{F})}, \widetilde{\mathrm{dg}(\mathsf{Inj}(\mathfrak{F}))})\] of exact chain complexes and complexes with relatively injective terms, as defined in \cite[Definition 4.2.18]{kelly2016homotopy} is a cotorsion pair. To see that it is complete, one simply dualises the proof of \cite[Theorem 4.3.58]{kelly2016homotopy}. Here the presence of enough injectives and countable completeness is used to prove the existence of \(\mathrm{dg}\)-injective resolutions. By \cite[Theorem 4.1.7]{kelly2016homotopy}, the pair \((\mathsf{Infl}(\widetilde{\mathrm{Ob}(\mathfrak{F})}), \mathsf{Defl}(\widetilde{\mathrm{dg}(\mathsf{Inj}(\mathfrak{F}))}))\) defined by degreewise inflations in \(\mathsf{Ch}(\mathfrak{F})\) with exact cokernels, and deflations with fibrant kernels, is a compatible weak factorisation system. It is equal to the pair \((\mathcal{C} \cap \mathcal{W}, \mathcal{F}')\). Since weak factorisation systems satisfy left and right lifting properties with respect to each other, we have that \(\mathcal{F}' = \mathcal{F}\) as required. 
\end{proof}

\begin{proposition}\label{prop:model_pro-kom-general}
Let \(\mathfrak{F}\) be a countably complete exact category with enough injectives. Then there is a closed model category structure on \(\mathsf{Kom}(\mathfrak{F})\) where the weak equivalences are the quasi-isomorphisms, the cofibrations the degreewise inflations, and the fibrations the cokernels of cofibrations with fibrant kernels. 
\end{proposition}

\begin{proof}
In the proof of Lemma \ref{lem:fibration-description}, we have already seen that the existence of enough injectives on \(\mathfrak{F}\) implies that the pair \((\widetilde{\mathrm{Ob}(\mathfrak{F})}, \widetilde{\mathrm{dg}(\mathsf{Inj}(\mathfrak{F}))})\) is a complete cotorsion pair on \(\mathsf{Ch}(\mathfrak{F})\). Furthermore, the pair \((\widetilde{\mathrm{dg}(\mathrm{Ob}(\mathfrak{F}))}, \widetilde{\mathsf{Inj}(\mathfrak{F}}))\) coincides with the injective cotorsion pair on \(\mathsf{Ch}(\mathfrak{F})\). This is because \(\widetilde{\mathsf{Inj}(\mathfrak{F}})\) coincides with split exact chain complexes of injective  objects relative to the exact category structure on \(\mathfrak{F}\), which in turn equals the collection of injective objects on \(\mathsf{Ch}(\mathfrak{F})\), by dualising \cite[Proposition 2.6.111]{kelly2016homotopy}. The collection \(\widetilde{\mathrm{dg}(\mathrm{Ob}(\mathfrak{F}))}\) equals \(\mathsf{Ch}(\mathfrak{F})\), using \cite[Proposition 4.2.53]{kelly2016homotopy}. And, since the presence of enough injectives on \(\mathfrak{F}\) implies the same for \(\mathsf{Kom}(\mathfrak{F})\) (by adapting the proof of \cite[Corollary 2.6.112]{kelly2016homotopy}), the cotorsion pair \((\mathsf{Ch}(\mathfrak{F}), \widetilde{\mathsf{Inj}(\mathfrak{F})})\) is complete. Finally, for the class \(\overline{\mathcal{W}}\) of exact chain complexes,  we have \(\widetilde{\mathrm{dg}(\mathrm{Ob}(\mathfrak{F}))} \cap \overline{\mathcal{W}} = \mathsf{Ch}(\mathfrak{F}) \cap \overline{\mathcal{W}} = \overline{\mathcal{W}} = \widetilde{\mathrm{Ob}(\mathfrak{F})}\) and \(\widetilde{\mathrm{dg}(\mathsf{Inj}(\mathfrak{F}))} \cap \overline{\mathcal{W}} = \widetilde{\mathsf{Inj}(\mathfrak{F})}\), where the last identity follows from \cite[4.2.34]{kelly2016homotopy}. The Hovey Correspondence (see \cite[Theorem 3.3]{gillespie2011model} for the exact categorical version) induces the model structure as in the statement of the proposition. Finally, by \cite[5.2.4]{kelly2016homotopy}, \(\mathsf{Kom}(\mathfrak{F})\) inherits the same  model structure, by interpreting weak equivalences, cofibrations and fibrations degree-wise. 
\end{proof}

We now specialise Proposition \ref{prop:model_pro-kom-general} to our setting. Concretely, given a quasi-abelian category \(\mathcal{E}\) and a full subcategory \(J\), we want the locally split exact category structure on \(\overleftarrow{\mathcal{E}}\) relative to \(J\) to induce a model category structure on \(\mathsf{Kom}(\overleftarrow{\mathcal{E}})\) by interpreting chain maps and extensions degreewise. Note that there is a related category, namely, the category \(\overleftarrow{\mathsf{Kom}(\mathcal{E})}\) of projective systems of complexes with entries in \(\mathcal{E}\). Proposition \ref{lem:pro-com-pro} below shows that these two categories are equivalent. First we require the following technical lemma.

\begin{lemma}\label{lem:pro-technical}
Let \(R\) be a ring that is finitely generated as a \(\Z\)-module. Let \((X,\sigma_X)\) and \((Y,\sigma_Y)\) be projective systems of \(R\)-modules, and \(f \colon X \to Y\) a morphism of projective systems of \(\Z\)-modules that is \(R\)-linear in the sense that 

$$\bfig
  \square/>`>`>`>/[R \otimes_\Z X ` X ` R \otimes_\Z Y `Y;m_X` 1 \otimes f ` f `m_Y]
\efig$$

\noindent commutes, where \(m_X\) and \(m_Y\) are the multiplication maps of \(X\) and \(Y\), respectively. Then \(f\) can be represented as a morphism of projective system of \(R\)-modules. 
\end{lemma}

\begin{proof}
We can represent \(f\) by \(\Z\)-linear maps \((f_n \colon X_{m(n)} \to Y_n)_{n \in \N}\). The condition of \(R\)-linearity says that for each generator \(r \in R\), there are indices \(k_{r}(n) \geq m(n)\) to have \(f_n(\sigma^X(r \cdot x)) = r \cdot \sigma^Y(f_{k_{r}}(x))\) for \(x \in X_{k_{r}(n)}\). Now since \(R\) is finitely generated, we can arrange that this equality holds simultaneously on all the generators of \(R\), by taking the maximum \(k(n)\) of all such indices \(k_r(n)\). So \(\{ f_n\circ\sigma\colon X_{k(n)} \to Y_n\}\) is a morphism of projective systems of \(R\)-modules.  
\end{proof}

\begin{proposition}\label{lem:pro-com-pro}
Let \(\mathcal{E}\) be an additive category with cokernels. Then we have an equivalence of categories \(\overleftarrow{\mathsf{Kom}(\mathcal{E})} \cong \mathsf{Kom}(\overleftarrow{\mathcal{E}})\).
\end{proposition}

\begin{proof}

Let \(R\) be the ring with the presentation \(\setgiven{g,d}{g^2 = 1, gd + dg = 0, d^2 = 0}\). Then a \(\Z/2\Z\)-graded chain complex over \(\mathcal{E}\) is equivalent to an object \(X \in \mathcal{E}\), together with a ring homomorphism \(R \to \mathrm{End}_{\mathcal{E}}(X)\). Now since \(\mathcal{E}\) is additive and has cokernels, so does \(\overleftarrow{\mathcal{E}}\). Therefore any chain complex in \(\mathsf{Kom}(\overleftarrow{\mathcal{E}})\) is a projective system \(X = (X_n)_{n\in \N}\) in \(\overleftarrow{\mathcal{E}}\), together with a ring homomorphism \(f \colon R \to \mathrm{End}_{\overleftarrow{\mathcal{E}}}(X)\). Viewing \(R\) as an \(R\)-module, we obtain a projective system \(R \otimes_\Z X\) of \(R\)-modules. Since \(R\) is finitely generated and free as an abelian group and, since \(\overleftarrow{\mathcal{E}}\) is an additive category, \(R \otimes_\Z X\) is well-defined. It is a direct sum of finitely many copies of \(X\). The map \(f\) induces a morphism \(R \otimes_\Z X \to X\) of projective systems in \(\overleftarrow{\mathcal{E}}\). Tensoring on the left with \(R\), we obtain a morphism \(R \otimes_\Z R \otimes_\Z X \to R \otimes_\Z X\) of projective systems in \(\overleftarrow{\mathcal{E}}\), which is \(R\)-linear. Since \(R\) is finitely generated, Lemma \ref{lem:pro-technical} implies that we can represent this map as a projective system of \(R\)-module maps \(R \otimes_\Z R \otimes_\Z X \to R \otimes_\Z X\). These can be further represented as a diagram of \(R\)-modules \((R \otimes_\Z R \otimes_\Z X_n \to R \otimes_\Z X_n)_{n\in M}\), after suitably reindexing by some directed set \(M\), with \((X_n)_{n\in M} \cong X\) in \(\overleftarrow{\mathcal{E}}\). For each \(n\), the cokernel of \(R \otimes_\Z R \otimes_\Z X_n \to R \otimes_\Z X_n\) is \(X_n\), so that each \(X_n\) is an \(R\)-module. Therefore, \(X\) is a projective system of \(R\)-modules.  By naturality of the bar resolution, this assignment is indeed a functor \(\mathsf{Kom}(\overleftarrow{\mathcal{E}}) \to \overleftarrow{\mathsf{Kom}(\mathcal{E})}\), which is inverse to the functor \(\overleftarrow{\mathsf{Kom}(\mathcal{E})} \to \mathsf{Kom}(\overleftarrow{\mathcal{E}})\) that forgets the \(R\)-action on a diagram in \(\mathcal{E}\). 
\end{proof}

Now given \(C\), \(D \in \mathsf{Kom}(\overleftarrow{\mathcal{E}}) \cong \overleftarrow{\mathsf{Kom}(\mathcal{E})}\), there are two internal mapping spaces, namely, \(\mathrm{HOM}_{\overleftarrow{\mathcal{E}}}(C,D)\) and \(\mathrm{Hom}_{\overleftarrow{\mathsf{Kom}(\mathcal{E})}}(C,D) = \varinjlim_n \varprojlim_m \mathrm{HOM}_{\mathcal{E}}(C_n, D_m).\) Since the forgetful functor in Proposition \ref{lem:pro-com-pro} is fully faithful, we have a bijection \(\mathrm{HOM}_{\overleftarrow{\mathcal{E}}}(C,D) \cong \mathrm{Hom}_{\overleftarrow{\mathsf{Kom}(\mathcal{E})}}(C,D)\). In the situation where \(\cE\) is quasi-abelian and \(J\) is a full subcategory, we use this bijection and the locally split exact category structure on \(\overleftarrow{\mathcal{E}}\), to construct a model structure on \(\overleftarrow{\mathsf{Kom}(\mathcal{E})}\) by means of cofibrations and weak equivalences defined on \(\mathsf{Kom}(\overleftarrow{\mathcal{E}})\). For the rest of this article, \(\mathrm{HOM}_{\overleftarrow{\mathcal{E}}}\) will unambiguously denote the morphism set of the category \(\overleftarrow{\mathsf{Kom}(\mathcal{E})}\). In particular, we call a morphism \(f \colon X \to Y\) in \(\overleftarrow{\mathsf{Kom}(\cE)} \cong \mathsf{Kom}(\overleftarrow{\cE})\) a \textit{cofibration} if it is degree-wise an inflation for the locally split exact structure on \(\overleftarrow{\cE}\).

We now spell out the fibrations in the category \(\overleftarrow{\mathsf{Kom}(\mathcal{E})} \cong \mathsf{Kom}(\overleftarrow{\cE})\) more explicitly when \(\overleftarrow{\cE}\) has the locally split exact structure relative to a full subcategory \(J\) of a quasi-abelian category \(\cE\). A pro-complex \(X = (X_n)_{n \in \N} \in \overleftarrow{\mathsf{Kom}(\mathcal{E})} \cong \mathsf{Kom}(\overleftarrow{\mathcal{E}})\) is \textit{fibrant} if at each degree, \(X_n\) is relatively injective for the locally split exact category structure on \(\overleftarrow{\mathcal{E}}\) relative to \(J\).  Finally, a morphism \(f \colon X \to Y\) in \(\overleftarrow{\mathsf{Kom}(\cE)}\) is a \textit{fibration} if it is degree-wise a deflation and \(\ker(f)\) is fibrant.

We now describe the injective model structure on the category \(\overleftarrow{\mathsf{Kom}(\mathsf{Ind}(\mathcal{C}))}\) (which by Proposition \ref{lem:pro-com-pro} is the same as the category \(\mathsf{Kom}(\overleftarrow{\mathsf{Ind}(\mathcal{C})})\).

\begin{theorem}\label{thm:existence-injective}
Let \(\mathcal{C}\) be a quasi-abelian category, and consider \(\overleftarrow{\mathsf{Ind}(\cC)}\) as an exact category with respect to the locally split exact structure relative to \(\cC\). Then the injective model structure exists on \(\overleftarrow{\mathsf{Kom}(\mathsf{Ind}(\mathcal{C}))} \cong \mathsf{Kom}(\overleftarrow{\mathsf{Ind}(\cC)})\). Explicitly,
\begin{itemize}
\item its weak equivalences are the quasi-isomorphisms for the exact category structure on \(\overleftarrow{\mathsf{Ind}(\mathcal{C})}\);
\item its cofibrations are degree-wise inflations for the exact structure above;
\item its fibrations are degree-wise deflations for the exact structure above, with fibrant kernels. 
\end{itemize}
\end{theorem}

\begin{proof}
By Theorem \ref{thm:enough-injectives-pro}, under the hypotheses on \(\cC\), the category \(\overleftarrow{\mathsf{Ind}(\cC)}\) with the locally split exact structure relative to \(\cC\) has enough injectives. The existence of finite limits in \(\mathsf{Ind}(\cC)\) implies that \(\overleftarrow{\mathsf{Ind}(\cC)}\) has countable limits.  Proposition \ref{prop:model_pro-kom-general} now yields the desired result.  
\end{proof}

The Hovey correspondence mentioned in the proof of Proposition \ref{prop:model_pro-kom-general} also provides that the collections \(\widetilde{\mathrm{dg}(\mathrm{Ob}(\mathfrak{F}))} = \mathsf{Kom}(\mathfrak{F})\) (resp. \(\widetilde{\mathrm{Ob}(\mathfrak{F})}\)) and \(\widetilde{\mathrm{dg}(\mathsf{Inj}(\mathfrak{F}))}\) (resp. \(\widetilde{\mathsf{Inj}(\mathfrak{F})}\)) are the \textit{cofibrant} (resp. \textit{trivially cofibrant}) and \textit{fibrant} (resp. \textit{trivially fibrant}) objects of \(\mathsf{Kom}(\mathfrak{F})\), respectively. The \textit{trivial objects} are, of course, the exact chain complexes. The nomenclature ``injective'' model structure is due to the fact that the trivially fibrant objects coincide with the injective objects of \(\mathsf{Kom}(\mathfrak{F})\). 

We now describe the quasi-isomorphisms and exact chain complexes in this category more explicitly. 

\begin{definition}
  \label{def:locally_contractible}
  Let \(\mathcal{C}\) be an additive category with kernels and cokernels and let \(C=(C_k,\alpha_n^k)_{k,n\in\N}\) be a chain complex over
  \(\overleftarrow{\mathsf{Ind}(\mathcal{C})}\).  We may arrange for
  each~\(C_k\) to be a chain complex and write
  \(C_k \cong (C_{k,i})_{i\in I_k}\) as an inductive system of chain
  complexes.  For each \(n,k\in\N\), \(i\in I_k\), with \(k\ge n\),
  let \(\alpha_{n,i}^k \colon C_{k,i} \to C_n\) be the component of
  the structure map \(\alpha_n^k\colon C_k \to C_n\) of the
  projective system at~\(i\); this is a morphism in \(\mathcal{C}\)
  to~\(C_{n,j}\) for some \(j\in I_n\).  The chain complex~\(C\) is
  called \emph{locally contractible} if, for every~\(n\), there is a
  \(k \geq n\) such that for any \(i \in I_k\), the
  map~\(\alpha_{n,i}^k\) is null-homotopic.  A chain map
  \(f \colon C \to D\) is called a \emph{local chain homotopy
    equivalence} if its mapping cone is locally contractible.
\end{definition}

\begin{proposition}
  \label{prop:locally_contractible_exact}
A chain complex in \(\overleftarrow{\mathsf{Ind}(\cC)}\) is  locally contractible if and only if it is exact for the locally split exact structure on \(\overleftarrow{\mathsf{Ind}(\cC)}\).
\end{proposition}

\begin{proof}
  Let~\(C\) be a locally contractible chain complex.  Write
  \(C \cong (C_k,d_k)_{k \in \Z}\) with a compatible family of
  morphisms \(d_k \colon C_k \to C_k\) in
  \(\mathsf{Ind}(\cC)\) with \(d_k^2=0\), as in the
  definition of a locally contractible chain complex.  Then
  \(\ker(d) \cong \ker(d_n)_{n\in \Z}\).  We need to prove that the
  morphism of projective systems described by the morphisms
  \(d_n \colon C_n \to \ker(d_n)\) is a cokernel in the
  category \(\overleftarrow{\mathsf{Ind}(\cC)}\).
  Let \((C_{k,i},d_{k,i})\), \(\alpha_n^k\) and \(\alpha_{n,i}^k\) be as in
  Definition~\ref{def:locally_contractible}.  Let \(n\in\Z\).
  Since~\(C\) is locally contractible, there is \(k \geq n\) such
  that for each \(i \in I_k\), there is a map
  \(h_{n,i}^k \colon C_{k,i} \to C_n\) with
  \[
    h_{n,i}^k \circ d_{k,i} + d_n \circ h_{n,i}^k = \alpha_{n,i}^k.
  \]
  We replace~\(h_{n,i}^k\) by its restriction to \(\ker d_{k,i}\), which
  satisfies \(d_n \circ h_{n,i}^k = \alpha_{n,i}^k\).  Composing
  with the structure maps~\(\alpha_k^l\), we get such maps for all
  \(l\ge k\) and \(i\in I_l\) as well.  For \(l\ge k\), we build a
  pull-back diagram
  
$$\bfig
  \square/>`>`>`>/[X_{l,n} ` \ker(d_l) ` C_n  ` \ker(d_n) ; g_{l,n}` \gamma_{l,n} ` \alpha_n^l `d_n]
\efig$$

  The universal property of pullbacks gives a unique map
  \(\sigma_{l,i}^n \colon \ker(d_l)_i \to X_{l,n}\) with
  \(g_{l,n}\circ \sigma_{l,n}^i = \mathrm{can}_i\colon \ker(d_l)_i
  \to \ker(d_l)\) and
  \(\gamma_{l,n} \circ \sigma_{l,n}^i = h_{n,i}^l\).  Then
  \(g_{l,n} \colon X_{l,n} \to \ker(d_l)\) is a cokernel in
  \(\mathsf{Ind}(\cC)\).  The maps~\((g_{l,n})\)
  combine to a morphism of pro-ind systems.  This morphism is a
  cokernel because each~\(g_{l,n}\) is a cokernel.  Since the family
  of maps \(\ker (d_l) \to \ker (d_n)\) represents the identity map
  of projective systems, \(X\) is isomorphic as a projective system
  to~\(C\), and the maps~\((g_{l,n})\) represent the map
  \(d\colon C\to \ker(d)\).  Consequently,
  \(d \colon C \to \ker(d)\) is a cokernel.
  
 To see the converse, let \((C,d)\) be an exact chain complex in \(\overleftarrow{\mathsf{Ind}(\cC)}\). Then by definition, \(\ker(d) \into C \onto \ker(d)\) is a locally split extension. Now the proof of \cite[Theorem 3.3.9]{thesis} applies to yield local contracting homotopies for the projective system structure maps of \(C\).   
\end{proof}

The following lemma describes local chain homotopy equivalences
directly without referring to the mapping cone:

\begin{proposition}
  \label{prop:explicit_local_homotopy_equivalences}
  Let \(f \colon C \to D\) be a chain map in
  \(\overleftarrow{\mathsf{Ind}(\cC)}\).  We may
  represent~\(f\) by a compatible family
  \((f_n \colon C_n \to D_n)_{n \in \N}\) of chain maps in
  \(\mathsf{Ind}(\cC)\), and each~\(f_n\) by a coherent
  family of chain maps \(f_{n,i}\colon C_{n,i} \to D_{n,i}\)
  in~\(\cC\) for \(i\in I_n\) with some filtered category
  ~\(I_n\).  Then \(f\) is a local chain homotopy equivalence if and only if for each \(n \in \N\), there is an
  \(m \geq n\), such that for each \(i \in I_m\), there are
  morphisms
  \[
    g_{m,i}^n \colon D_{m,i} \to C_n,\qquad
    h_{m,i}^D \colon D_{m,i} \to D_n[1],\qquad
    h_{m,i}^C \colon C_{m,i} \to C_n[1],
  \]
  where~\(g_{m,i}^n\) are chain maps and \(h_{m,i}^D\) and
  \(h_{m,j}^C\) are chain homotopies between \(f_n \circ g_{m,i}^n\)
  and \(g_{m,i} \circ f_{m,i}\), and the canonical maps
  \(\eta_{m,i}^n \colon D_{m,i} \to D_n\) and
  \(\gamma_{m,i}^n \colon C_{m,i} \to C_n\), respectively.
\end{proposition}

\begin{proof}
  We need to show that \(\mathsf{cone}(f)\) is locally contractible,
  that is, for each \(n\), there is an \(m \geq n\) such that for
  all \(i \in I_m\), the structure map
  \begin{equation}
    \label{eq:to_be_made_null-homotopic}
    \mathsf{cone}(f)_{m,i}^n = C[-1]_{m,i} \oplus D_{m,i}
    \xrightarrow{\gamma_{m,i}^n \oplus \eta_{m,i}^n}
    C[-1]_n \oplus D_n = \mathsf{cone}(f)_n
  \end{equation}
  is null-homotopic.  Here \(\gamma\) and~\(\eta\) are the structure
  maps of \(C\) and~\(D\), respectively.
  Let~\(\delta^{\mathsf{cone}(f)_n}\) denote the boundary map of the
  cone of~\(f_n\).  Since \(h_{m,i}^C\), \(h_{m,i}^D\) are local
  chain homotopies between \(g_{m,i} \circ f_{m,i}\) and
  \(\gamma_{m,i}^n\), and \(f_n \circ g_{m,i}\) and
  \(\eta_{m,i}^n\), respectively, the matrix
  \[
    \tilde{h}_{m,i} = \begin{pmatrix}
      -h_{m,i}^{C[-1]} & g_{m,i}^n\\
      0 & h_{m,i}^D
    \end{pmatrix}
    \colon \mathsf{cone}(f)_{m,(i,j)} \to \mathsf{cone}(f)_n
  \]
  satisfies
  \[
    \delta^{\mathsf{cone}(f)_n} \circ \tilde{h}_{m,i} + \tilde{h}_{m,i}
    \circ \delta^{\mathsf{cone}(f)_{m,(i,j)}} \\
    =  \begin{pmatrix}
      \gamma_{m,i}^n &  h_{m,i}^D \circ f_{m,i} - f_n \circ h_{m,i}^C \\
      0 & \eta_{m,i}^n
    \end{pmatrix}.
  \]
  Then we compute that \(h = \tilde{h} \circ \Psi\) with
  \[
    \Psi_{m,i}^n \defeq \begin{pmatrix}
      \eta_{m,i}^n &  f_n \circ h_{m,i}^C - h_{m,i}^D \circ f_{m,i}\\
      0 & \gamma_{m,i}^n
    \end{pmatrix},
  \]
  is the desired null-homotopy for~\eqref{eq:to_be_made_null-homotopic}.
\end{proof}

\section{Bivariant local and analytic cyclic homology}\label{sec:bivariant}
\numberwithin{equation}{subsection}

In this section, we specialise Theorem \ref{thm:existence-injective} to the case where \(\cC = \mathsf{Ban}_k\) is the category of Banach spaces and bounded \(k\)-linear maps over a nontrivially valued Banach field \(k\).  By \cite[Lemma A.30]{Ben-Bassat-Kremnizer:Nonarchimedean_analytic}, this category is quasi-abelian. Furthermore, it is a symmetric monoidal category with respect to the completed projective tensor product and \(k\) is its tensor unit. It is also \textit{closed} in the sense that its internal Hom-object \(\Hom_{\mathsf{Ban}_k}(A,B)\)  is a Banach space with respect to the operator norm. It is also finitely complete and cocomplete, so that its ind-completion \(\mathsf{Ind}(\mathsf{Ban}_F)\) is bicomplete. As a consequence of these properties, categories of inductive systems of Banach spaces, and the full subcategory \(\mathsf{CBorn}_k\) of complete bornological vector spaces are ideal for the purposes of cyclic homology theories for topological algebras. 

\subsection{Local cyclic homology for nonarchimedean Banach algebras}\label{subsec:nonahc}

In this subsection, we shall see that something peculiar happens when the category \(\cC\) in Proposition \ref{prop:locally-split-specialised} is the category of  Banach spaces over a discretely valued nonarchimedean Banach field - the ind-locally split exact structure trivialises to the quasi-abelian structure. Let $V$ be a complete discrete valuation ring, $F$ its fraction field and $\mathbb{F}$ its residue field. For the purposes of non-archimedean cyclic theories, we are mainly interested in the additive category of \(\Z/2\Z\)-graded complexes in the category \(\overleftarrow{\mathsf{Ind}(\mathsf{Ban}_F)}\) of projective systems of inductive systems of Banach \(F\)-vector spaces. In this section, we explictly describe the injective model structure - that is, the cofibrations and weak equivalences from the previous section - which greatly simplifies for the category \(\overleftarrow{\mathsf{Ind}(\mathsf{Ban}_F)}\).

\begin{lemma}\label{lem:Ban-enough-injectives}
The quasi-abelian structure on the category \(\mathsf{Ban}_F\) coincides with its split exact structure. In particular, it has enough projectives. Furthermore, \(\mathsf{Ind}(\mathsf{Ban}_F)\) has enough injectives for its quasi-abelian structure. 
\end{lemma}

\begin{proof}
Let \(K \into E \overset{q}\onto Q\) be any extension of Banach spaces. Then by \cite[Remark 10.2]{Schneider:Nonarchimedean}, \(Q \cong C_0(D,F)\) for some set \(D\). Now by \cite[A. 38]{Ben-Bassat-Kremnizer:Nonarchimedean_analytic}, \(C_0(D,F)\) is projective for the quasi-abelian structure on \(\mathsf{Ban}_F\), so that the quotient map \(E \onto Q\) splits. As a consequence, the quasi-abelian and the split exact structures coincide in the category \(\mathsf{Ban}_F\).   Finally, \(\mathsf{Ind}(\mathsf{Ban}_F)\) has enough injectives by Lemmas \ref{lem:ind-C-enough-injectives} and \ref{lem:ind-enough-inj}.
\end{proof}

\begin{theorem}\label{thm:locally-split-quasi-abelian}
Every extension in \(\mathsf{Ind}(\mathsf{Ban}_F)\) is ind-locally split. Consequently, any extension \(K \into E \onto Q\) in the category \(\overleftarrow{\mathsf{Ind}(\mathsf{Ban}_F)}\) is locally split relative to \(\mathsf{Ban}_F\).
\end{theorem}

\begin{proof}
Let \(K \into E \onto Q\) be an extension in \(\mathsf{Ind}(\mathsf{Ban}_F)\), represented by an inductive system of extensions \((K_i \into E_i \onto Q_i)_{i \in I}\). By Lemma \ref{lem:Ban-enough-injectives}, extension in the system splits in \(\mathsf{Ban}_F\). Consequently, if \(X \to Q\) is any morphism represented by a level map \(X \to Q_i\), its composition with the section at that level yields a map \(X \to E_i\). Composing with the map \(E_i \to E\) yields the required lifting of the original map \(X \to Q\). For the second part, any extension \(K \into E \onto Q\) can be represented by a diagram of extensions \((K_n \into E_n \onto Q_n)_{n \in \N}\) of objects in \(\mathsf{Ind}(\mathsf{Ban}_F)\). For each fixed \(n\), by the first part, the extension \(K_n \into E_n \onto Q_n\) is already ind-locally split. Hence \(K \into E \onto Q\) is locally split, by Propostion  \ref{prop:locally-split-specialised}.
\end{proof}

\begin{corollary}\label{cor:pro-ind-enough-inj}
The category \(\overleftarrow{\mathsf{Ind}(\mathsf{Ban}_F)}\) has enough injectives for its quasi-abelian structure, which coincides with its locally split exact structure relative to \(\mathsf{Ban}_F\). 
\end{corollary}

\begin{proof}
By Corollary \ref{cor:ind-C-enough-injectives}, \(\mathsf{Ind}(\mathsf{Ban}_F)\) has enough injectives for the ind-locally split exact structure, which coincides with the quasi-abelian structure by Theorem \ref{thm:locally-split-quasi-abelian}. The conclusion now follows from Lemma \ref{lem:enough-injectives}. 
\end{proof}

What we have therefore shown is that locally split extensions in \(\overleftarrow{\mathsf{Ind}(\mathsf{Ban}_F)}\) relative to \(\mathsf{Ban}_F\) are equivalent to projective limits of diagrams of \emph{all} extensions in \(\mathsf{Ind}(\mathsf{Ban}_F)\). The surprising feature of the nonarchimedean setting is that such extension in \(\mathsf{Ind}(\mathsf{Ban}_F)\) is equivalent to an ind-locally split extension in \(\mathsf{Ind}(\mathsf{Ban}_F)\) relative to the subcategory \(\mathsf{Ban}_F\).

In \cite{Cortinas-Meyer-Mukherjee:NAHA, Meyer-Mukherjee:localhc, Meyer-Mukherjee:HA}, the authors define three chain complex valued functors \[\mathbb{HA} \colon \overleftarrow{\mathsf{Alg}(\mathsf{CBorn}_{V}^{\mathrm{tf}})} \to \mathsf{Der}(\overleftarrow{\mathsf{Ind}(\mathsf{Ban}_F)}),\] \[\mathbb{HA} \colon \overleftarrow{\mathsf{Alg}(\mathsf{Mod}_\mathbb{F})} \to \mathsf{Der}(\overleftarrow{\mathsf{Ind}(\mathsf{Ban}_F)}),\]
and \[\mathbb{HL} \colon \mathsf{Alg}_V^\dagger \to \mathsf{Der}(\overleftarrow{\mathsf{Ind}(\mathsf{Ban}_F))}\] for projective systems of complete, bornologically torsionfree \(V\)-algebras, projective systems of \(\mathbb{F}\)-algebras and \textit{dagger algebras}. The latter class of algebras were introduced in \cite{Meyer-Mukherjee:Bornological_tf}. The definition of these homology theories is beyond the scope of this article, and we therefore direct the interested reader to their original references above. Each of these functors is homotopy invariant for suitable classes of homotopies, matricially stable and excisive. 

An important property of the analytic cyclic homology theory defined in \cite{Meyer-Mukherjee:HA} is that it is independent of choices of liftings. More precisely, the main result that used local homotopy equivalences that this article clarifies conceptually is the following:

\begin{theorem}\cite[Theorem 5.5]{Meyer-Mukherjee:HA}\label{thm:main}
Let \(D_1\) and \(D_2\) be two dagger algebras such that \(D_1/\pi D_1 \cong A \cong D_2/\pi D_2\), and the quotient maps are bounded when we view \(A\) as a bornological algebra with the fine bornology. Then we have weak equivalences \(\mathbb{HA}(D_1) \simeq \mathbb{HA}(A) \simeq \mathbb{HA}(D_2)\). 
\end{theorem}

In other words, the weak equivalences of the injective model structure on the category \(\mathsf{Kom}(\overleftarrow{\mathsf{Ind}(\mathsf{Ban}_F)})\) are precisely the local chain homotopy equivalences used to prove \cite[Theorem 5.5]{Meyer-Mukherjee:HA}. The authors also define bivariant versions of analytic cyclic homology in \cite{Cortinas-Meyer-Mukherjee:NAHA, Meyer-Mukherjee:HA}, whose morphism (ind-Banach) space we may now interpret using the homotopy category of our model category: let \(A\) and \(B\) belong to one of the categories of algebras mentioned at the start of the section. Denoting by \(\mathrm{Hom}_{\mathsf{Ho}(\overleftarrow{\mathsf{Ind}(\mathsf{Ban}_F)})}\) the morphism space of the homotopy category of the model category \(\overleftarrow{\mathsf{Ind}(\mathsf{Ban}_F)}\) with its injective model structure, \textit{bivariant analytic cyclic homology} can be redefined as \begin{align*}
    \mathrm{HA}_i(A,B) \defeq \mathsf{Hom}_{\mathsf{Ho}(\mathsf{Kom}(\overleftarrow{\mathsf{Ind}(\mathsf{Ban}_F)}))}(\mathbb{HA}(A), \mathbb{HA}(B)[i]) \\
    \cong H_i(\mathsf{Hom}_{\mathsf{HoKom}(\overleftarrow{\mathsf{Ind}(\mathsf{Ban}_F)})}(R\mathbb{HA}(A), R\mathbb{HA}(B)))
\end{align*} for \(i = 0,1\), where \(R\) is the fibrant replacement functor. Note that by the description of cofibrant objects provided after Theorem \ref{thm:existence-injective}, every object is already cofibrant, so there is no need for cofibrant replacement.  As local cyclic homology is defined only by changing the bornology on a dagger algebra to the \textit{compactoid bornology} (see \cite[Section 3]{Meyer-Mukherjee:localhc}), that is, \(\mathbb{HL}(A) = \mathbb{HA}(A')\), where \(A'\) is the dagger algebra \(A\) with the compactoid bornology, we may also put \[\mathrm{HL}_i(A,B) \defeq \mathrm{Hom}_{\mathsf{Ho}(\mathsf{Kom}(\overleftarrow{\mathsf{Ind}(\mathsf{Ban}_F)}))}(\mathbb{HL}(A), \mathbb{HL}(B)[i]),\] for \(i = 0,1\). By \cite[Theorem 7.4]{Meyer-Mukherjee:localhc}, we have a chain homotopy equivalence \(\mathbb{HL}(V) \cong \mathbb{HA}(V)\), so that \[\mathrm{HL}_i(V, B) \cong \mathrm{HL}_i(B)\] for all dagger algebras \(B\) and \(i = 0,1\).

\subsection{Local cyclic homology for pro-\(C^*\)-algebras}\label{subsec:prohl}

We now consider the category \(\mathsf{Ban}_\C\) of complex Banach spaces, viewed as an exact category for its split exact structure. This does \emph{not} simplify to the quasi-abelian structure. The resulting ind-locally split exact structure on the category \(\mathsf{Ind}(\mathsf{Ban}_\C)\) is used to define the \textit{local homotopy category of complexes} - the correct target category of \textit{local cyclic homology} (\cite{Meyer:HLHA,Puschnigg:Diffeotopy}). By definition, the local homotopy category of chain complexes is the localisation of the naive homotopy category of \(\mathsf{Kom}(\mathsf{Ind}(\mathsf{Ban}_\C))\) of chain complexes at the collection of chain maps \(f\) whose mapping cone \(\mathsf{cone}(f)\) is ind-locally split exact.  It is actually rather important to work in this generality to prove an important property of local cyclic homology, namely, invariance under \textit{isoradial embeddings} defined in \cite[Section 3.4]{Meyer:HLHA}. For an isoradial, dense subalgebra \(A \subseteq B\) - for instance \(\mathrm{C}^\infty(M) \subseteq \mathrm{C}(M)\) for a smooth manifold \(M\) - the induced map \(\mathbb{HL}(\mathrm{C}^\infty(M)) \to \mathbb{HL}(\mathrm{C}(M))\) in local cyclic homology is a local chain homotopy equivalence by \cite[Theorem 6.21]{Meyer:HLHA}. Equivalently, its cone is ind-locally split exact. 

Invariance under isoradial embeddings is an important reason why local cyclic homology yields good results for \(C^*\)-algebras. However, neither local cyclic homology nor the related \textit{analytic cyclic homology}, commute with inverse limits, which prevents their extension to pro-\(C^*\)-algebras.  To extend local cyclic homology to pro-\(C^*\)-algebras, one needs to enlarge the target category to include projective systems of chain complexes in \(\mathsf{Ind}(\mathsf{Ban}_\C)\). This is what justifies the generality of Theorem \ref{thm:existence-injective}. With the homotopy category of \(\overleftarrow{\mathsf{Kom}(\mathsf{Ind}(\mathsf{Ban}_\C))}\) with the injective model structure as the target, we can extend analytic cyclic homology

\[
\xymatrix@R+1pc@C+2pc{
\mathsf{Alg}(\mathsf{CBorn}_\C) \ar[d] \ar[r]^{\mathbb{HA}} & \mathsf{Ho}(\overleftarrow{\mathsf{Kom}(\mathsf{Ind}(\mathsf{Ban}_\C))}) \\
\overleftarrow{\mathsf{Alg}(\mathsf{CBorn}_\C)} \ar[ur]_{\mathbb{HA}^{\mathrm{pro}}} & }
\] to projective systems of complete bornological \(\C\)-algebras, by applying it levelwise to each bornological algebra. Similarly, we can extend local cyclic homology \[\mathbb{HL}^{\mathrm{pro}} \colon \overleftarrow{\mathsf{Alg}(\mathsf{Ind}(\mathsf{Ban}_\C))} \to \mathsf{Ho}(\overleftarrow{\mathsf{Kom}(\mathsf{Ind}(\mathsf{Ban}_\C))})\] to pro-algebras.

\begin{theorem}\label{thm:pro-HL-properties}

The functors
\begin{gather*} 
\mathbb{HA}^{\mathrm{pro}} \colon \overleftarrow{\mathsf{Alg}(\mathsf{CBorn}_\C)} \to \mathsf{Ho}(\overleftarrow{\mathsf{Kom}(\mathsf{Ind}(\mathsf{Ban}_\C))})\\
\text{ and } \mathbb{HL}^{\mathrm{pro}} \colon \overleftarrow{\mathsf{Alg}(\mathsf{Ind}(\mathsf{Ban}_\C))} \to \mathsf{Ho}(\overleftarrow{\mathsf{Kom}(\mathsf{Ind}(\mathsf{Ban}_\C))})
\end{gather*} 
satisfy 
\begin{itemize}
\item homotopy invariance for homotopies of bounded variation;
\item stability with respect to algebras of nuclear operators;
\item excision for extensions of pro-bornological \(\C\)-algebras with a bounded pro-linear section. 
\end{itemize}

\end{theorem}

\begin{proof}
We view the space of bounded variations \(\mathcal{A}([0,1])\) and the algebra \(\mathcal{l}\) of nuclear operators as constant pro-systems. Now for each \(n\), we have chain homotopy equivalences \(\mathbb{HA}(A_n) \simeq \mathbb{HA}(A_n \otimes_{\pi} \mathcal{A}([0,1]))\) and \(\mathbb{HA}(A_n) \simeq \mathbb{HA}(A_n \otimes_{\pi} \mathcal{M})\) by \cite[Theorem 5.45, Theorem 5.65]{Meyer:HLHA}. By varying \(n\), we get weak equivalences \(\mathbb{HA}^{\mathrm{pro}}(A) \simeq \mathbb{HA}^{\mathrm{pro}}(A \otimes_{\pi} \mathcal{A}([0,1])\) and \(\mathbb{HA}^{\mathrm{pro}}(A) \simeq \mathbb{HA}(A \otimes_{\pi} \mathcal{M} )\) in the homotopy category. For excision, we observe that any extension of pro-algebras that splits by a pro-linear section can be represented by an extension of bornological algebras with compatible bounded linear sections. Now use the excision theorem for bornological algebras with bounded linear sections. 
\end{proof}

We remark that Theorem \ref{thm:pro-HL-properties} goes through even if we work in the quasi-abelian category \(\overleftarrow{\mathsf{Ind}(\mathsf{Ban}_\C)}\) as the equivalences \(\mathbb{HA}(A) \simeq \mathbb{HA}(A \otimes_{\pi} \mathcal{A}([0,1])\), \(\mathbb{HA}(A) \simeq \mathbb{HA}(A \otimes_{\pi} \mathcal{M})\) are chain homotopy equivalences, rather than the more general weak equivalences of Theorem \ref{thm:existence-injective}. The context in which these more general local chain homotopy equivalences are unavoidable is the following:

\begin{theorem}\label{thm:isoradial-pro}
Let \(A \to B\) be a pro-algebra homomorphism in \(\overleftarrow{\mathsf{Ind}(\mathsf{Alg}(\mathsf{Ban}_\C))}\) that is represented by an inverse system of isoradial embeddings at each pro-system level. Then \(\mathbb{HL}^{\mathrm{pro}}(A) \simeq \mathbb{HL}^{\mathrm{pro}}(B)\) in  \(\mathsf{Ho}(\overleftarrow{\mathsf{Kom}(\mathsf{Ind}(\mathsf{Ban}_\C))})\). 
\end{theorem}

\begin{proof}
By definition of \(\mathbb{HL}^{\mathrm{pro}}\), we only need to consider the levelwise maps \[\mathbb{HL}(A_n) \to \mathbb{HL}(B_n)\] between local cyclic homology complexes in \(\mathsf{Ind}(\mathsf{Ban}_\C)\), the cones of which are ind-locally split exact relative to \(\mathsf{Ban}_\C\) as a consequence of \cite[Theorem 6.20 and 6.11]{Meyer:HLHA}.  
\end{proof}

\subsubsection{Chern character}\label{subsub:chern}
    Let $CP$ be the category whose objects are the separable $C^*$-algebras and whose morphisms are the completely positive maps, and $C^*\mathsf{Alg}\subset CP$ the subcategory with the same objects but where homomorphisms are $*$-algebra homomorphisms. 
    Let $\cC\subset\overleftarrow{CP}$ be the full subcategory on the inverse systems whose transition maps are surjective $*$-algebra homomorphisms, and let $\cD\subset \cC$ the category with the same objects, and with the homomorphisms making it into a full subcategory of $\overleftarrow{C^*\mathsf{Alg}}$.
    In his thesis \cite[Chapter 3]{Bonkat:Thesis}, Bonkat constructs a functor \(\cD \to \mathbf{KK}_{\cC}^{\cD}\) and proves \cite[Satz 3.5.11]{Bonkat:Thesis} that it is universal among those taking valued in an additive category that are \(\mathrm{C}([0,1])\)-homotopy invariant, compact operator stable and half exact with respect to extensions with completely positive contractive linear sections. Actually, the construction in \cite{Bonkat:Thesis} works for general pairs of subcategories \(\cD \subseteq \cC\) of pro-\(C^*\)-algebras, satisfying the axioms laid out in \cite[Section 2.1 and Definition 2.4.1]{Bonkat:Thesis}. 
    Now by Theorem \ref{thm:pro-HL-properties}, the restriction of \(HL_0^{\mathrm{pro}}(-,-)\) to \(\mathcal{D}\) satisfies all these properties, as the \(C^*\)-algebra of compact operators are a special case of algebras of nuclear operators. So by the universal property of Bonkat's bivariant \(K\)-theory,  we get a bivariant Chern character \[\mathbf{KK}_{\cC}^{\cD}(A,B) \to \mathrm{HL}_0^{\mathrm{pro}}(A,B) = \mathrm{Hom}_{\mathsf{Ho}(\mathsf{Kom}(\overleftarrow{\mathsf{Ind}(\mathsf{Ban}_\C)}))}(\mathbb{HL}^{\mathrm{pro}}(A), \mathbb{HL}^{\mathrm{pro}}(B)),\] for pro-$C^*$-algebras \(A\) and \(B \in \cD\).

\end{document}